\documentclass[a4paper]{amsart}
\usepackage[utf8]{inputenc}  
\usepackage{currfile}
\usepackage{graphicx}
\usepackage{upref}
\usepackage{url}
\usepackage[pdftex, colorlinks=true]{hyperref}
\usepackage{color}
\definecolor{darkred}{rgb}{0.6,0,0}

\newcommand{\Dt}{{\Delta t}}
\newcommand{\abs}[1]{\left\vert#1\right\vert}

\newcommand{\N}{\mathbb N}
\newcommand{\Z}{\mathbb Z}
\newcommand{\norm}[1]{\left\Vert#1\right\Vert}

\newcommand{\test}{\varphi}
\newcommand{\eps}{\varepsilon}
\DeclareMathOperator*{\sgn}{sign}
\newcommand{\sign}[1]{\sgn\left(#1\right)}
\newcommand{\dott}{\, \cdot\,}

\newcommand{\marginlabel}[1]%
       {\mbox{}\marginpar{\raggedleft\hspace{0pt}\tiny{\textcolor{red}{#1}}}}

\newcommand{\Dlp}{\Delta_+}
\newcommand{\Dlm}{\Delta_-}

\newcommand{\ob}[1]{\overline{#1}}
\newcommand{\indic}{\mathbb I}
\newcommand{\pt}{\partial_t}

\newcommand{\bigOh}{\mathcal{O}}

\newtheorem{theorem}{Theorem}[section]

\newtheorem{lemma}[theorem]{Lemma}

\newtheorem{remark}[theorem]{Remark}
\newtheorem{remark*}[theorem]{Remark}
\newtheorem{corollary}[theorem]{Corollary}

\numberwithin{equation}{section}     
\allowdisplaybreaks

\begin{document}
\title[Non-local Follow-the-Leader]{The continuum limit of non-local Follow-the-Leader
models}

\author[Holden]{Helge Holden}
\address[Holden]{\newline
    Department of Mathematical Sciences,
    NTNU Norwegian University of Science and Technology,
    NO--7491 Trondheim, Norway}
\email[]{\href{helge.holden@ntnu.no}{helge.holden@ntnu.no}} 
\urladdr{\href{https://www.ntnu.edu/employees/holden}{https://www.ntnu.edu/employees/holden}}

\author[Risebro]{Nils Henrik Risebro}
\address[Risebro]{\newline
 Department of Mathematics,
University of Oslo,
  P.O.\ Box 1053, Blindern,
  NO--0316 Oslo, Norway }
\email[]{\href{nilshr@math.uio.no}{nilshr@math.uio.no}} 
\urladdr{\href{https://www.mn.uio.no/math/english/people/aca/nilshr/index.html}{https://www.mn.uio.no/math/english/people/aca/nilshr/index.html}}
\date{\today} 



\thanks{Both authors were supported by the Swedish Research Council under grant no. 2021-06594 while they were in residence at Institut Mittag-Leffler in Djursholm, Sweden during the fall semester of 2023. The research of HH was supported in part by the project \textit{IMod --- Partial differential equations, statistics and data: An interdisciplinary approach to data-based modelling}, project number 65114, from the Research Council of Norway. }


\begin{abstract} 
We study a generalized Follow-the-Leader model where the driver considers the position of an arbitrary but finite number of vehicles ahead, as well as  the position of the vehicle directly behind the driver. It is proved that this model converges to the classical Lighthill--Whitham--Richards model for traffic flow when traffic becomes dense.  This also underscores the robustness of the Lighthill--Whitham--Richards model.
\end{abstract}

\maketitle
\section{Introduction} \label{sec:intro}

We study a generalized Follow-the-Leader (FtL) model for unidirectional traffic of $M_\ell$ vehicles on a single lane road given by
\begin{align}
\dot x_i(t)&= \sum_{j=0}^N c_j v\bigg(\frac{\ell}{x_{i+j+1}(t)-x_{i+j}(t)}\bigg) \notag \\
&\quad+ \kappa \Bigg(v\bigg(\frac{\ell}{x_{i+1}(t)-x_{i}(t)}\bigg)-v\bigg(\frac{\ell}{x_{i}(t)-x_{i-1}(t)}\bigg) \Bigg), \quad i=1,\dots,M_\ell. \label{eq:MAIN}
\end{align}
Here the position of the $i$th vehicle, each of length $\ell$, is $x_i(t)$ at time $t$. The system \eqref{eq:MAIN} is closed by posing appropriate periodic boundary conditions, see later.

The velocity function $v$ is a decreasing function that vanishes at maximum capacity of the road. Each driver considers the distances to the $N$ vehicles ahead and the one vehicle right behind. The impact of more distant vehicles is less pronounced,  and thus we assume that the constants $c_j$ decrease, i.e.,  $c_j\ge c_{j+1}\ge 0$.  In addition, the driver considers the distance to the vehicle right behind, and if that becomes too short, the driver will speed up, and the coefficient $\kappa$ measures this influence. 

We show that as the length of each vehicle becomes smaller, $\ell\to 0$  and the number of vehicles increases, $M_\ell\to \infty$, while $N$ is kept fixed, the distribution of vehicles will approach the solution of the classical Lighthill--Whitham--Richards (LWR) model \cite{LW_II,richards}
\begin{equation} \label{eq:LWR}
\rho_t+(\rho v(\rho))_x=0,
\end{equation}
where the density, or rather saturation,  $\rho$ is approximated by  $\ell/(x_{i+1}(t)-x_{i}(t))$.  This is of course a scalar hyperbolic conservation law \cite{HoldenRisebro}.  The result is independent of the finer details given by $N$ and the coefficients $c_j$, $\kappa$, and shows the robustness of the LWR model.   As long as $N$ remains unaltered in the limiting process, this limit is distinct from the widely studied problem of non-local to local limit for conservation law, see, e.g., \cite{MR4651679,MR4362534}.

The model \eqref{eq:MAIN} is based  on the following anticipated behavior: In general, the longer the distance to the vehicle in front, the faster the drivers are willing to drive.  However, each driver considers the distance between successive vehicles ahead of the driver.  More weight is given to the vehicles close to the driver. By the same token, the driver can look in the rear mirror and assess the distance to the vehicle immediately behind. If that distance is shorter than the distance to the vehicle just in front, the driver will speed up. This generalization of the FtL model is more realistic than the traditional one, as it takes into account finitely many vehicles ahead of the driver, and includes the observed fact that also the behavior of the vehicle right behind you influences your actions.

More specifically, we are given coefficients $c_j\ge 0$ that indicate the weight given to the velocity between the $j$th and $(j+1)$th vehicle (as counted from the  $i$th vehicle). We assume that the influence drops with the distance, thus  $c_j\ge c_{j+1}\ge 0$. The driver is willing to consider $N$ vehicles.  We have $\sum_{j=0}^N c_j=1$, and for convenience we put $c_N=0$. Similarly, if the distance to the vehicle immediately behind is shorter than the distance to the vehicle directly in front, the driver will speed up. The impact of this is scaled by the numerical parameter $\kappa>0$. This results in the model \eqref{eq:MAIN}.

There are two main classes of mathematical models for traffic flow, namely discrete models based on car-following on the one hand, and, on the other hand, continuum models based on the assumption of dense traffic for which the flow can be described by a density, resulting in  ``traffic hydrodynamics'' models. 
The dichotomy between microscopic and macroscopic models, or discrete and continuum models, is of course one of the fundamental outstanding problems of mathematical physics. The problem here is a considerably easier than the general problem, however, it allows for a rigorous analysis of the limit, as the number ``particles'' tends to infinity. 

There is a wide range of car-following models dating from the late fifties and early sixties, see \cite{Brackstone}. We here generalize a Follow-the-Leader model based on what is called safety-distance models or collision avoidance models, as a feature of this model is that it is collision-free. On the other hand, the LWR model \cite{LW_II,richards} has been, and still is, the prevalent continuum model. A consequence of the analysis in the present paper is that  even for the generalized FtL model presented here, the scaling limit remains the LWR model, and this offers yet another justification for the LWR model. 

By now there are several ways to show that the standard FtL model converges to the LWR model, the first one being \cite{FrancescoRosini}. The realization that the Follow-the-Leader is nothing but a semi-discrete approximation of the LWR in Lagrangian coordinates simplified the proof considerably, see \cite{holris18a,holris18b}.   See also \cite{Argall_etal,AwKlarMaterneRascle,ColomboRossi,1605.05883,rosini, GoatinRossi,Rossi,FrancescoStivaletta} and references therein. There are also proofs in the setting of traffic on a network \cite{colombomarcelliniholden,CristianiSahu}.

\medskip
Let us now describe the content of this paper. We introduce the short-hand notation $\Delta_\pm a_i=\pm(a_{i\pm 1}-a_i)$ and $\bar a_i=\sum_{j=0}^N c_j a_{i+j}$. If we define $v_i=v(\ell/(x_{i+1}(t)-x_{i}(t)))=v(\ell/\Delta_+x_i(t))$, the model \eqref{eq:MAIN} takes the compact form
\begin{equation}
\dot x_i(t)= \bar v_i+\kappa \Delta_-v_i. \label{eq:MAIN_short}
\end{equation}
It turns out that it is convenient to introduce the Lagrangian variable $y_i(t)=\Delta_+x_i(t)/\ell$, and then we find the following equation
\begin{equation}
\ell\dot y_i(t)= \Delta_+\bar V_i+\kappa \Delta_+\Delta_-V_i, \label{eq:MAIN_shortL}
\end{equation}
with $V(y)=v(1/y)$, and hence $V_i=V(y_i)=v_i$.  Note that for $N=1$ we recover the traditional FtL model. 

The proof  that the FtL model converges to the LWR model proceeds as follows: To avoid technical complications associated with boundary terms, we consider the periodic case. More precisely, we impose in equation \eqref{eq:MAIN_shortL} that
\begin{equation*}
y_{i+M_\ell}=y_i, \quad i \in \Z.
\end{equation*}

We first analyze the equation for $y_i$. By introducing a spatially piecewise constant function $y_\ell$ using the values $y_i$ and a fixed grid in space of size $M_\ell=1/\ell$, cf.~\eqref{eq:lattice_y}, we can show that 
\begin{equation*}
\abs{y_{\ell}(t,\dott)}_{BV([0,1])}\le \abs{y_{\ell}(0,\dott)}_{BV([0,1])}, \quad   \norm{y_\ell(t,\dott)-y_\ell(s,\dott)}_{L^1([0,1])}\le C \abs{t-s}.
\end{equation*}
Furthermore, if we have another solution $z_i(t)$ of \eqref{eq:MAIN_shortL}, we get stability in the sense that
\begin{equation*}
\norm{y_\ell(t,\dott)- z_\ell(t,\dott)}_{L^1([0,1])}\le \norm{y_\ell(0,\dott)- z_\ell(0,\dott)}_{L^1([0,1])}.
\end{equation*}
This suffices to obtain  strong convergence  $y_\ell\ \to \ y$ in $C([0,T]; L^1([0,1]))$ as $\ell\to 0$.  However, as we are interested in the
Eulerian formulation in terms of the density $\rho$, we simply define  $\rho_i(t)=1/y_i(t)$, and derive the corresponding equation for  $\rho_i(t)$, and translate the properties from $y_i$ to $\rho_i$. We introduce the function  $\rho_\ell(t,x)=\sum_{i=1}^{M_\ell} \rho_i(t) \indic_i(t,x)$ where  
$\indic_i(t,x)$ is the indicator function of the time-dependent
spatial interval $[x_i(t),x_{i+1}(t))$.  Since $y_i$ is periodic, the
corresponding $\rho_i$ will be periodic in  Eulerian coordinates, with some period $P$, see equation \eqref{eq:x-periodic}.
We ensure that we stay away from vacuum by assuming that $\rho_0\ge \nu>0$. Then we can prove, cf.~Lemma \ref{lem:convergence},
 \begin{align*}
    \inf \rho_0 \le \rho_\ell(t,x) &\le \sup\rho_0,\\
    \abs{\rho_\ell(t,\dott)}_{BV([0,P])}&\le \frac1{\nu^2}\abs{\rho_0}_{BV([0,P])}, \\
    \norm{\rho_\ell(t,\dott)-\rho_\ell(s,\dott)}_{L^1([0,P])}
    &\le C \abs{t-s}.
  \end{align*}
Observe that we get the somewhat unexpected constant $1/\nu^2$ in the estimate for bounded variation. This is the case provided $N>1$; in the classical case $N=1$ the constant is replaced by  unity. 

These estimates establish the existence of a limit  $\rho\in C([0,T];L^1([0,P]))\cap L^\infty([0,T];BV([0,P]))$, cf.~Corollary \ref{cor:rhoconv}, as $\ell\to 0$. Here $L^\infty([0,T];BV([0,P]))$ denotes the set of functions $u=u(t,x)$ with
\begin{equation*}
\sup_{t\in[0,T]}\abs{u(t,\dott)}_{BV([0,P])}<\infty.
\end{equation*}
It remains to show that the limit equals the unique weak entropy solution of the LWR equation \eqref{eq:LWR}, which means that it satisfies, cf.~Theorem \ref {lem:weaksolN}, 
\begin{equation*}
      \int_0^\infty\int_0^P \big(\eta(\rho)
    \test_t +q(\rho)\test_x\big) \,dxdt
      +\int_0^P \eta(\rho_0)\test(0,x)\,dx\ge 0
  \end{equation*}
  for any non-negative $P$-periodic test function $\test\in C^\infty_c([0,\infty)\times[0,P])$.   Here $\eta$ is a convex (entropy) function and  $q$ is the entropy flux, satisfying $q'(\rho)=\eta'(\rho)(\rho v(\rho))'$.   This shows that $\rho$ is indeed the unique weak entropy solution of the LWR equation.

\section{The model}  \label{sec:model}

Consider $M_\ell$ identical vehicles on a unidirectional, single lane
road with initial positions $x_1(0) <x_2(0) < \cdots < x_{M_\ell}(0)$
where $x_{i+1}(0)-x_i(0)>\ell$, with $\ell>0$ being the length of each
vehicle. The velocity $v$ is assumed to be a decreasing Lipschitz function of
a single variable. The ``non-localness'' enters the model in the
following way.

Given constants $c_j$ for $j=0,\dots,N$
\begin{equation}
\sum_{j=0}^N c_j=1, \quad c_0\ge \cdots \ge c_{N-1}\ge 0, \  c_N=0,
\end{equation}
we define for any sequence $a_i$
\begin{equation}
\bar a_i=\sum_{j=0}^N c_j a_{i+j}.
\end{equation}
Furthermore, we define the traditional shift operators as follows
\begin{equation}
\Delta_\pm a_i=\pm(a_{i\pm 1}-a_i).
\end{equation}
Observe that we have, by applying summation by parts, that
\begin{equation}
\begin{aligned}
  \Dlp \ob{a}_i
  = \ob{\Dlp{a}_i}
  &=\sum_{j=0}^N c_j\left(a_{i+j+1}-a_{i+j}\right) =-c_0 a_i + \sum_{j=1}^N \left(c_{j-1}-c_j\right)a_{i+j} \\
  &=\sum_{j=1}^N\left(c_j-c_{j-1}\right)\left(a_i-a_{i+j}\right) 
 =-\sum_{j=1}^N \Dlm c_j \left(a_{i+j}-a_i\right).
\end{aligned}\label{eq:byparts}
\end{equation}

Given a Lipschitz continuous non-increasing velocity function $v\colon[0,1] \to [0,1]$ with $v(0)=1$ and  $v(1)=0$, we assume that the dynamics of the $i$th vehicle is given by 
\begin{align}
\dot x_i(t)&= \sum_{j=0}^N c_j v\big(\frac{\ell}{x_{i+j+1}(t)-x_{i+j}(t)}\big) \notag \\
&\quad+ \kappa \Big(v\big(\frac{\ell}{x_{i+1}(t)-x_{i}(t)}\big)-v\big(\frac{\ell}{x_{i}(t)-x_{i-1}(t)}\big) \Big), \quad i\in Z_\ell, \label{eq:MAINmodel}
\end{align}
with $Z_\ell=\{1,\dots,M_\ell\}$, and $\kappa$ is a positive constant.  
With the introduced notation 
we can write \eqref{eq:MAINmodel} compactly as
\begin{equation}
\dot x_i(t)= \bar v_i+\kappa \Delta_-v_i,\label{eq:MAIN_shortmodel}
\end{equation}
with $v_i=v(\ell/\Delta_+x_i(t))$.  To avoid technicalities connected with boundary conditions, we assume periodicity. More concretely, we assume the existence of a positive $P$ such that 
\begin{equation}\label{eq:x-periodic}
\text{$x_{i+M_\ell}(t)=x_i(t)+P$ for all $i\in\Z$ and all $t$}.
\end{equation}

As in \eqref{eq:MAIN_shortL} we write  $y_i(t)=\Delta_+x_i(t)/\ell$, which implies
that \eqref{eq:MAIN_shortmodel} takes the form
\begin{equation}
\ell\dot y_i(t)= \Delta_+\ob{V_i}+\kappa \Delta_+\Delta_-V_i, \quad i\in Z_\ell, \label{eq:MAIN_shortLmodel}
\end{equation}
where $V(y)=v(1/y)$, and hence $V_i=V(y_i)=v_i$.   Note that
$y_{i+M_\ell}=y_i$ for all $i$. This gives a finite-dimensional system
of ordinary differential equations, and (local in time) existence of a unique solution follows from standard theory. Clearly, $V$ is increasing and $V\colon [1,\infty)\to [0,1]$. 

\section{The continuum limit}
Next we will study the limit when $\ell\to 0$. 

\subsection{Entropy estimates}
Inspired by conservation laws, let $(\eta,Q)$ be
an entropy/entropy flux pair, i.e., $\eta$ is twice continuously
differentiable and convex, and $Q$ is defined by $Q'=\eta' V'$.  We
multiply \eqref{eq:MAIN_shortLmodel} with $\eta'(y_i)$ and use   \eqref{eq:byparts} to obtain 
\begin{align*}
  \ell \frac{d}{dt}\eta_i
  &=\eta'(y_i) \Dlp\ob{V_i}+\kappa \eta'(y_i)\Dlp\Dlm V_i\\
  &=\Dlp\ob{Q_i} + \eta'(y_i)\Dlp\ob{V_i} - \Dlp\ob{Q_i} \\
   &\quad +\kappa \Dlp\left(\eta'(y_i)\Dlm V_i\right)-\kappa \left(\Dlp
    \eta'(y_i)\right)
    \left(\Dlp V_i\right)\\
  &= \Dlp\ob{Q_i} + \sum_{j=1}^N\left(c_j-c_{j-1}\right)
    \left[\left(\eta'(y_i)V_i - Q_i\right) -
    \left(\eta'(y_i)V_{i+j} - Q_{i+j}\right)\right]
  \\
  &\quad  +\kappa \Dlp\left(\eta'(y_i)\Dlm V_i\right)-\kappa \left(\Dlp
    \eta'(y_i)\right)
    \left(\Dlp V_i\right)\\ 
  &=\Dlp\ob{Q_i} + \sum_{j=1}^N\left(c_j-c_{j-1}\right)
    H(y_i,y_{i+j})\\
  &\quad
     +\kappa \Dlp\left(\eta'(y_i)\Dlm V_i\right)-\kappa \left(\Dlp
    \eta'(y_i)\right)
    \left(\Dlp V_i\right),
\end{align*}
where
\begin{align*}
  H(a,b)
  &=\left[ \left(\eta'(a) V(a)-Q(a)\right)
    - \left(\eta'(a)V(b)-Q(b)\right)\right]\\
  &= \int^a_0 (\eta'(a)-\eta'(\sigma))V'(\sigma)\,d\sigma -
    \int^b_0 (\eta'(a)-\eta'(\sigma))V'(\sigma)\,d\sigma\\
  &=\int_a^b(\eta'(\sigma)-\eta'(a))V'(\sigma)\,d\sigma
    =\int_a^b \int_a^\sigma \eta''(\mu)\,d\mu V'(\sigma)\,d\sigma\ge
    0.
\end{align*}
Here we have written, in obvious notation, $\eta_i=\eta(y_i)$, $Q_i=Q(y_i)$.  Furthermore, we get
\begin{equation}
  \label{eq:entrbnd}
  \begin{aligned}
    \ell \frac{d}{dt}\eta_i-\sum_{j=1}^N\left(\Dlm c_j\right)
    H(y_i,y_{i+j})+&\kappa \left(\Dlp \eta'(y_i)\right)
    \left(\Dlp V_i\right) \\
    &= \Dlp\ob{Q_i}+\kappa \Dlp\left(\eta'(y_i)\Dlm V_i\right),
  \end{aligned}
\end{equation}
and since $c_j\le c_{j-1}$,  we see that the second term is non-negative.  Furthermore, since $\eta'$ and $V$ are both increasing functions, also the third term is non-negative. This entropy
equality immediately implies the entropy inequality
\begin{equation}
  \label{eq:entrineq}
  \ell \frac{d}{dt}\eta_i\le  \Dlp\ob{Q_i}+\kappa \Dlp\left(\eta'(y_i)\Dlm V_i\right),
\end{equation}
and by an approximation argument, this is valid for any Lipschitz
continuous convex entropy $\eta$. Hence
\begin{equation}\label{eq:entropydecrease}
  \frac{d}{dt}\ell\sum_{i\in Z_\ell} \eta(y_i(t)) \le 0.
\end{equation}
Choosing
\begin{equation*}
  \eta(y)=\left(y-\inf_i y_i(0)\right)^-\ \ \ \text{and}\ \ \
  \eta(y)=\left(y-\sup_i y_i(0)\right)^+,
\end{equation*}
where $a^\pm=(\abs{a}\pm a)/2$, implies that
\begin{equation*}
  \inf_{i\in Z_\ell} y_i(0)\le y_i(t)\le \sup_{i\in Z_\ell} y_i(0),
\end{equation*}
for any positive $t$. Incidentally, this shows that the systems of
ordinary differential equations, \eqref{eq:MAIN_shortmodel} and \eqref{eq:MAIN_shortLmodel}, both have unique
global solutions for $t\in (0,\infty)$. Furthermore, it shows that  the model does not allow for collisions. 

Consider next another solution $z_i(t)$ of \eqref{eq:MAIN_shortLmodel} with initial data $z_i(0)$.
Subtract the equation for $z_i$ from the corresponding equation for $y_i$, and multiply by $\sign{y_i-z_i}$ to get
\begin{align}
  \ell \frac{d}{dt}\abs{y_i-z_i}
  &=\ell\sign{y_i-z_i}\frac{d}{dt}\left(y_i-z_i\right)\notag\\
  &=\sign{y_i-z_i}\Dlp\big(\ob{V(y_i)-V(z_i)}\big)\notag\\
  &\quad+\kappa\sign{y_i-z_i}\Dlm\Dlp\left(V(y_i)-V(z_i)\right)\notag\\
  &=\sign{y_i-z_i}\sum_{j=1}^N\left(c_j-c_{j-1}\right)
    \left[\left(V(y_i)-V(z_i)\right) - \left(V(y_{i+j})-V(z_{i+j})\right)
    \right]\notag\\
  &\quad + \kappa\Dlp\left(\sign{y_i-z_i}\Dlm(V(y_i)-V(z_i))\right)
\notag  \\
  &\quad
    -\kappa\left(\Dlp\sign{y_i-z_i}\right)\left(\Dlp(V(y_i)-V(z_i))\right)\notag\\
  &\le\sum_{j=1}^N\left(c_j-c_{j-1}\right)
    \left[\abs{V(y_i)-V(z_i)} - \abs{V(y_{i+j})-V(z_{i+j})}
    \right] \notag\\
  &\quad + \kappa\Dlp\left(\sign{y_i-z_i}\Dlm(V(y_i)-V(z_i))\right)\notag\\
  &=\Dlp\ob{\abs{V(y_i)-V(z_i)}} + \kappa\Dlp\left(\sign{y_i-z_i}\Dlm(V(y_i)-V(z_i))\right), \label{eq:L1stab}
\end{align}
where we used the fact that $c_j\le c_{j-1}$ to get the
inequality. In addition, we used that 
\begin{align*}
&\left(\Dlp\sign{y_i-z_i}\right)(\Dlp(V(y_i)-V(z_i))\\
   &\qquad\qquad=\left(\Dlp\sign{V(y_i)-V(z_i)}\right)\left(\Dlp(V(y_i)-V(z_i))\right)\ge 0,
\end{align*}
as both the signum function and $V$ are increasing functions. 
This immediately gives the stability estimate
\begin{equation}
  \label{eq:l1stab}
  \ell\sum_{i\in Z_\ell}\abs{y_i(t)-z_i(t)}\le \ell\sum_{i\in Z_\ell}\abs{y_i(0)-z_i(0)},
\end{equation}
for any $t>0$.  

Choosing $z_i = y_{i-1}$ yields a bound on the
variation of $y_i$, viz. 
\begin{equation}\label{eq:bvbnd}
  \sum_{i\in Z_\ell}\abs{y_i(t)-y_{i-1}(t)}\le
  \sum_{i\in Z_\ell}\abs{y_i(0)-y_{i-1}(0)}.
\end{equation}
It is convenient to view $y_i$ as a spatially $1$-periodic function. 
Define  for $i=0,\dots,M_\ell$  the quantities $\zeta_{i}=i/M_\ell$ and $I_i=[\zeta_{i},\zeta_{i+1})$. Then we let
\begin{equation}\label{eq:lattice_y}
  y_\ell(t,\zeta)=\sum_{i\in Z_\ell} y_i(t)\indic_{I_i}(\zeta),
\end{equation}
where $\indic_A$ is the indicator function of the set $A$.
Using this, \eqref{eq:l1stab} and \eqref{eq:bvbnd} read  $\norm{y_\ell(t,\dott)- z_\ell(t,\dott)}_{L^1([0,1])}\le \norm{y_\ell(0,\dott)- z_\ell(0,\dott)}_{L^1([0,1])}$ and $\abs{y_{\ell}(t,\dott)}_{BV([0,1])}\le \abs{y_{\ell}(0,\dott)}_{BV([0,1])}$, respectively.

Next we consider the time continuity of the approximate solution. Let now $t\ge s\ge 0$, and calculate 
\begin{align}
  \norm{y_\ell(t,\dott)-y_\ell(s,\dott)}_{L^1([0,1])}
  &=
  \ell\sum_{i\in Z_\ell}\abs{y_i(t)-y_{i}(s)} \notag\\
  &=\ell\sum_{i\in Z_\ell}\Bigl|\int_s^t\frac{d}{d\tau}
    y_i(\tau)\,d\tau\Bigr|\notag\\
  &\le \sum_{i\in Z_\ell}\int_s^t\Bigl|\ell\frac{d}{d\tau}
    y_i(\tau)\Bigr|\,d\tau\notag\\
  &=\sum_{i\in Z_\ell}\int_s^t
    \Bigl|\Dlp\ob{V(y_i(\tau))}+\kappa\Dlp\Dlm V(y_i(\tau))\Bigr|\,d\tau\notag\\
  &\le
    \norm{V'}_\infty\sum_{i\in Z_\ell}\int_s^t\abs{\Dlp\ob{y_i(\tau)}}+2\kappa
    \abs{\Dlp y_i(\tau)}\,d\tau\notag\\
  &\le \norm{V'}_\infty (1+2\kappa)\abs{y_{\ell}(0,\dott)}_{BV} (t-s), \label{eq:Lip}
\end{align}
so that the map $t\mapsto y_\ell(t,\dott)$ is Lipschitz continuous in $L^1$.  Here we used that
\begin{align}
\abs{\ob{a}}_{BV}= \sum_{i\in Z_\ell}\abs{\ob{\Dlp a_i}}&=  \sum_{i\in Z_\ell}\Bigl|\sum_{j=0}^N c_j \Dlp a_{i+j}\Bigr| 
\le  \sum_{i\in Z_\ell}\sum_{j=0}^N c_j \abs{\Dlp a_{i+j}} \notag \\
&=\sum_{j=1}^N  c_j \sum_{i\in Z_\ell}\abs{\Dlp a_{i+j}}
=\sum_{j=0}^N  c_j \abs{a}_{BV}=\abs{a}_{BV}. \label{eq:BVmean}
\end{align}
Note that the total variation is computed on $D_\ell$.

Recall that $\ell M_\ell=1$, from the estimates \eqref{eq:l1stab},
\eqref{eq:bvbnd}, and \eqref{eq:Lip} we can conclude the
strong convergence 
\begin{equation*}
  y_\ell\ \to \ y\ \ \ \text{in $C([0,T]; L^1([0,1]))$ as $\ell\to 0$},
\end{equation*}
 as $\ell\to0$,
see \cite[Thm.~A.8]{HoldenRisebro}.

At this point we could have shown that the limit $y$ is the unique entropy solution of the conservation
law $y_t - V_z=0$, subsequently transfer the result to Eulerian coordinates, and finally show that the corresponding density, or rather saturation,  is an entropy
solution of the LWR model. 

However, 
we shall do this directly for $\rho$ in the Eulerian setting. To this
end, we define $\rho_i(t)=1/y_i(t)$. Applying equation \eqref{eq:MAIN_shortLmodel} we find
\begin{equation}
  \label{eq:drhodef}
  \dot{\rho}_i = -\rho_i \Bigl(\frac{\Dlp \ob{v}_i}{\Dlp
    x_i}+\kappa\frac{\Dlp\Dlm v_i}{\Dlp x_i}\Bigr),
\end{equation}
for $t>0$, using  $\rho_i(t)=\ell/\Dlp x_i(t)$.  Observe the straightforward transition between the Lagrangian formulation \eqref{eq:MAIN_shortLmodel}  and the Eulerian formulation  \eqref{eq:drhodef}, sharply contrasting the cumbersome transition in the continuum case.

Next we define the appropriate initial data.  Assume that we
are given a non-negative $P$-periodic function $\rho_0\in L^1([0,P])\cap BV([0,P])$, such
that $\rho_0(x)\in [\nu,1]$ for all $x$, where $\nu$ is some small
positive number. This last assumption excludes vacuum.  

Choose
\begin{equation*}
M_\ell \in \N,\ \ \text{and define}\ \ \ell= \frac1{M_\ell}\int_0^P \rho_0(x)\, dx. 
\end{equation*}
Let  $x_0(0)=0$ and 
define $x_{i+1}(0)$ inductively by
\begin{equation}\label{eq:x0def}
  \int_{x_i(0)}^{x_{i+1}(0)} \rho_0(\xi)\,d\xi = \ell,\quad i=0,\dots, M_\ell-1.
\end{equation}
Next, we set
\begin{equation}\label{eq:rho0def}
  \rho_i(0)=\frac{\ell}{x_{i+1}(0)-x_i(0)}=:\frac{1}{y_i(0)}, \quad i=1,\dots,M_\ell,
\end{equation}
extending it periodically by $\rho_{i+M_\ell}(0)=\rho_i(0)$ for $i\in\Z$.  
Finally, we define $\rho_i(t)$ as the solution of \eqref{eq:drhodef} with initial data 
$\rho_i(0)$ defined by \eqref{eq:rho0def}.  
We can then
introduce
\begin{equation}\label{eq:rho_ell_def}
  \rho_\ell(t,x)=\sum_{i\in \Z} \rho_i(t) \indic_i(t,x),
\end{equation}
where $\indic_i(t,x)=\indic_{[x_i(t),x_{i+1}(t))}(x)$.  Note that $\indic_i$ and $\indic_{I_i}$ are distinct.   Observe that at this point the variable $y_i$ is superfluous; it is used only as a tool
to obtain the convergence. Converting the above calculations we get the following result.
\begin{lemma}
  \label{lem:convergence}
  Let $\rho_i$ for $i\in Z_\ell$ be defined by  \eqref{eq:drhodef}, and  $\rho_\ell(t,x)$ by \eqref{eq:rho_ell_def}.  Assume
  also that $\rho_\ell(0,x)=\rho_{\ell,0}(x) \ge\nu>0$ for all $\ell>0$ and $x$. Then  
  \begin{align}
    \label{eq:rhosup}
    \inf \rho_0
    &\le \rho_\ell(t,x) \le \sup\rho_0,\\
      \sum_{i\in Z_\ell} \abs{\rho_i(t)-\tilde\rho_i(t)}
    &\le \frac1{\nu^2} \sum_{i\in Z_\ell} \abs{\rho_{i,0}-\tilde\rho_{i,0}},\label{eq:rhoL1} \\
    \abs{\rho_\ell(t,\dott)}_{BV([0,P])}
    &\le\frac{1}{\nu^2}
      \abs{\rho_0}_{BV([0,P])}, \label{eq:rhobv}\\
    \norm{\rho_\ell(t,\dott)-\rho_\ell(s,\dott)}_{L^1([0,P])}
    &\le 2(1+2\kappa)\norm{v'}_\infty (t-s),\label{eq:rholip}
  \end{align}
  for all $0\le s\le t$. Here $\rho_\ell$ and $\tilde \rho_\ell$ are
  two solutions with $P$-periodic initial data $\rho_{\ell,0}$ and  $\tilde
  \rho_{\ell,0}$, respectively.
\end{lemma}
\begin{remark} Estimate  \eqref{eq:rhoL1} follows directly from the corresponding estimate  \eqref{eq:l1stab}.
However, this estimate cannot directly be expressed in terms of the $L^1$-norm in Eulerian 
coordinates.\footnote{We are grateful to Halvard O.~Storbugt for pointing this out.}
\end{remark}
\begin{proof}
  The inequalities \eqref{eq:rhosup}, \eqref{eq:rhoL1}, and
  \eqref{eq:rhobv} all follow from the corresponding inequalities for
  $y_\ell$, using that $\rho_i\ge\nu$.

  To prove \eqref{eq:rholip}, let $\omega_\eps$ be a standard
  mollifier and define
  \begin{equation}\label{eq:chi_eps}
    \chi^\eps_i(t,x)=\int_{-\infty}^x\big(
    \omega_\eps(\sigma-x_i(t))-\omega_\eps
    (\sigma-x_{i+1}(t))\big)\,d\sigma,
  \end{equation}
  and
  \begin{equation*}
    \rho^\eps_\ell(t,x)=\sum_i \rho_i(t) \chi_{i}^\eps(t,x).
  \end{equation*}
  Using that
  \begin{equation*}
    \frac{\partial}{\partial t} \chi^\eps_i(t,x)=
    \Dlp\left[\omega_\eps(x-x_i)\left(\ob{v}_i+\kappa\Dlm v_i\right)\right],
  \end{equation*}
  we get
  \begin{align*}
    \pt \rho^\eps_\ell(x,t)
    &=\sum_i \pt \left(\rho_i\chi^\eps_i(t,x)\right)\\
    &=\sum_i \dot{\rho_i}\chi^\eps_i(t,x)+\rho_i
      \frac{\partial}{\partial t}\chi^\eps_i(t,x)\\
    &=\sum_i-\rho_i \Bigl(\frac{\Dlp \ob{v}_i+\kappa\Dlp\Dlm v_i}{\Dlp
      x_i}\Bigr) \chi^\eps_i(t,x)+\rho_i
      \Dlp[\omega_\eps(x-x_i)(\ob{v}_i+\kappa\Dlm v_i)].
  \end{align*}
  Consequently
  \begin{align*}
    \norm{\rho^\eps_\ell(t,\dott)-\rho^\eps_\ell(s,\dott)}_{L^1([0,P])}
    &=\Bigl\| \int_s^t \pt
      \rho^\eps_\ell(\sigma,\dott)\,d\sigma\Bigr\|_{L^1([0,P])}\\
    &\le \int_s^t \int_0^P \Bigl|\pt
      \rho^\eps_\ell(\sigma, \dott)
      \Bigr|\,dxd\sigma\\
    &\le \int_s^t \sum_i \int_0^P \Big[\rho_i\frac{\abs{\Dlp
      \ob{v}_i+\kappa\Dlp\Dlm v_i}}{\Dlp
      x_i}\chi_i^\eps \\
    &\qquad\qquad
      +\rho_i \abs{\Dlp \left[\omega_\eps(x-x_i)(\ob{v}_i+\kappa\Dlm
      v_i)\right]}\Big]\,dx d\sigma.
  \end{align*}
  Now we can send $\eps$ to zero to obtain
  \begin{align*}
    \norm{\rho_\ell(t,\dott)-\rho_\ell(s,\dott)}_{L^1([0,P])}
    &\le \int_s^t \sum_i \Big[\rho_i \left(\abs{\Dlp \ob{v}_i} + \kappa
      \abs{\Dlp\Dlm v_i}\right) \\
    &\qquad\qquad+ \rho_i
      \left(\abs{\Dlp\ob{v}_i}+\kappa\abs{\Dlp\Dlm v_i}\right)\Big]
      \,d\sigma\\
    &\le 2\left(1+2\kappa\right)\norm{v'}_\infty (t-s), 
  \end{align*}
  where we have used that $\rho_i\le 1$ and \eqref{eq:BVmean}. 
\end{proof}
  
\begin{remark*}
  It is natural to ask whether \eqref{eq:rhosup}, \eqref{eq:rhoL1},
  and \eqref{eq:rhobv} can be proved directly from the scheme for
  $\rho_i$, i.e., from \eqref{eq:drhodef}. Using elementary
  techniques, this is easily accomplished for \eqref{eq:rhosup}. For
  \eqref{eq:rhoL1} and \eqref{eq:rhobv} one would hope to eliminate
  the constant $1/\nu^2$ (in fact, one would surmise that the constant would
  equal unity). However, we only managed to prove  \eqref{eq:rhoL1}
  and \eqref{eq:rhobv} with constant unity if $N=1$. To do this
  we rewrite the equation \eqref{eq:drhodef} for $\dot{\rho}_i$ as
  \begin{equation*}
   \ell \dot{\rho}_i = -\rho_i^2\Dlp(v_i+\kappa\Dlm v_i),
  \end{equation*}
  where we have used $\rho_i=\ell/\Dlp x_i$. Then we have
  \begin{align*}
  \ell  \frac{d}{dt}\abs{\rho_i-\rho_{i-1}}
    &=
      \sign{\Dlm\rho_i}\Dlm\dot{\rho}_i\\
    &=\sign{\Dlm v_i}\Dlm\left(\rho_i^2\Dlp v_i\right)+
      \kappa\sign{\Dlm v_i}\Dlm\left(\rho_i^2\Dlp\Dlm v_i  \right).
  \end{align*}
  We sum over $i\in Z_\ell$, and consider each term on the right
  separately. For the first term we get
  \begin{align*}
    \sum_i \sign{\Dlm v_i}\Dlm\left(\rho_i^2\Dlp v_i\right)
    &=\sum_i \sign{\Dlm v_i}\left(\rho_i^2\Dlp v_i - \rho_{i-1}^2\Dlm
      v_i\right)\\
    &=\sum_i \sign{\Dlm v_i}\rho_i^2\Dlp v_i-\rho_i^2\abs{\Dlp v_i}\le 0.
  \end{align*}
  As to the second term
  \begin{align}
    \sum_i\sign{\Dlm v_i}\Dlm\left(\rho_i^2\Dlp\Dlm v_i  \right)
    &=\sum_i \sign{\Dlm v_i}\left(\rho_i^2\Dlp\Dlm v_i -
      \rho_{i-1}^2\Dlm\Dlm v_i\right) \notag\\
    &=\sum_i\sign{\Dlm v_i}\Big(\rho_i^2\left(\Dlp v_i-\Dlm
      v_i\right)\notag \\
    &\qquad\qquad\qquad\qquad-
      \rho_{i-1}^2\left(\Dlm v_i-\Dlm v_{i-1}\right)\Big)\notag\\
    &=\sum_i \rho_i^2 \left[\Dlp v_i-\Dlm v_i\right]
      \left(\sign{\Dlm v_i}-\sgn{\Dlp v_i}\right)\le 0. \label{eq:bv_rho}
  \end{align}
  Thus
  \begin{equation*}
    \frac{d}{dt} \abs{\rho_\ell(t,\dott)}_{BV}\le 0,
  \end{equation*}
  and \eqref{eq:rhobv} holds with constant equal to $1$. We
  investigate the general case $N>1$ numerically. A natural numerical
  scheme consists in solving \eqref{eq:MAIN_shortmodel} by the Euler
  method, i.e., replacing \eqref{eq:MAIN_shortmodel} by
  \begin{equation}
    \label{eq:emicromod}
    \frac{1}{\Delta t}\big(x_i(t+\Delta t)-x_i(t)\big)
    =\ob{v_i(t)} + \kappa\Dlm v_i(t),
  \end{equation}
  where we must choose $\Dt\le \ell$ in order to avoid collisions. We
  choose $\kappa=0$, $N=10$, $c_j=1/10$ for $j=0,\ldots,4$,
  $\ell=1/45$, 
  $\Delta t=\ell$ and initial data given by
  \begin{equation*}
    \rho_0(x)=
    \begin{cases}
      1.0 &\abs{x}<0.5,\\ 0.05 &\text{otherwise,}
    \end{cases}
  \end{equation*}
  for $x$ in the interval $[-2,2]$, and extended periodically.
  \begin{figure}[h!]
    \centering
    \begin{tabular}[h!]{lr}
      \includegraphics[width=0.48\linewidth]{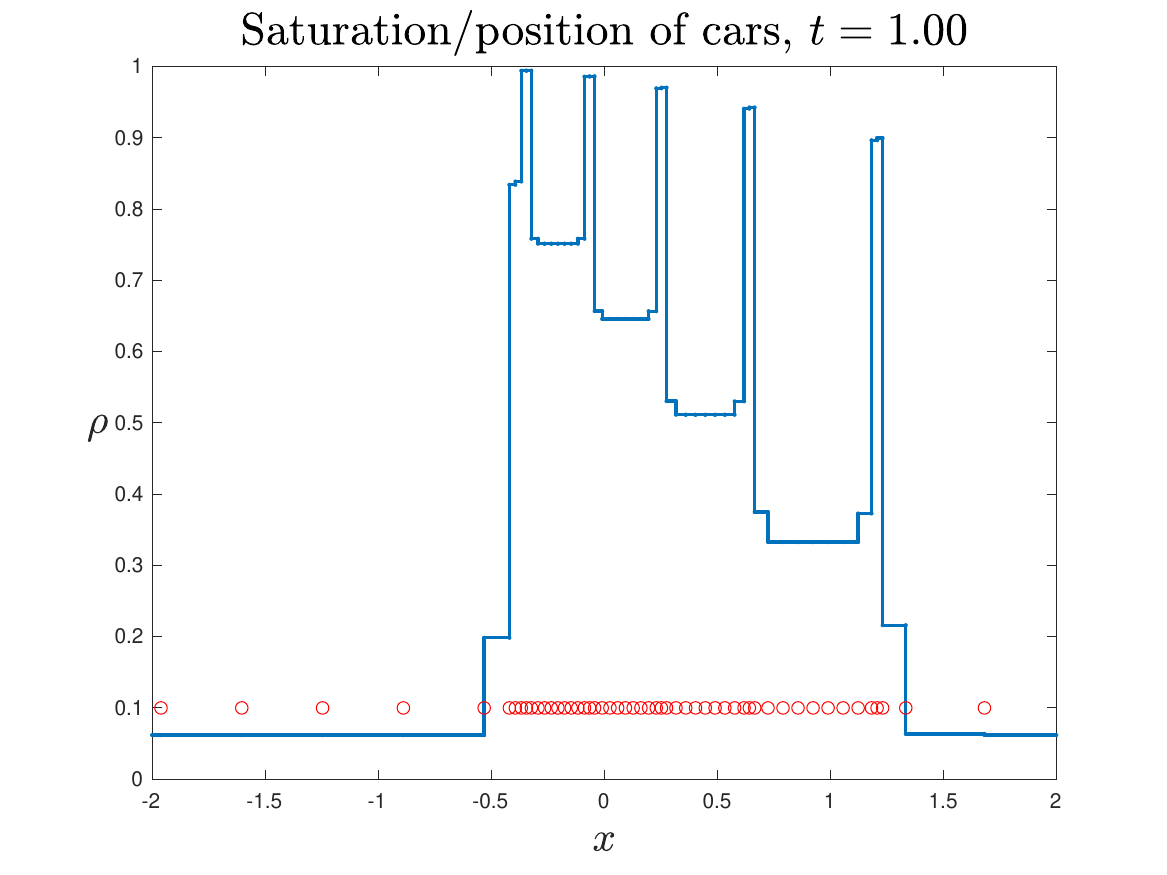}
      &\includegraphics[width=0.48\linewidth]{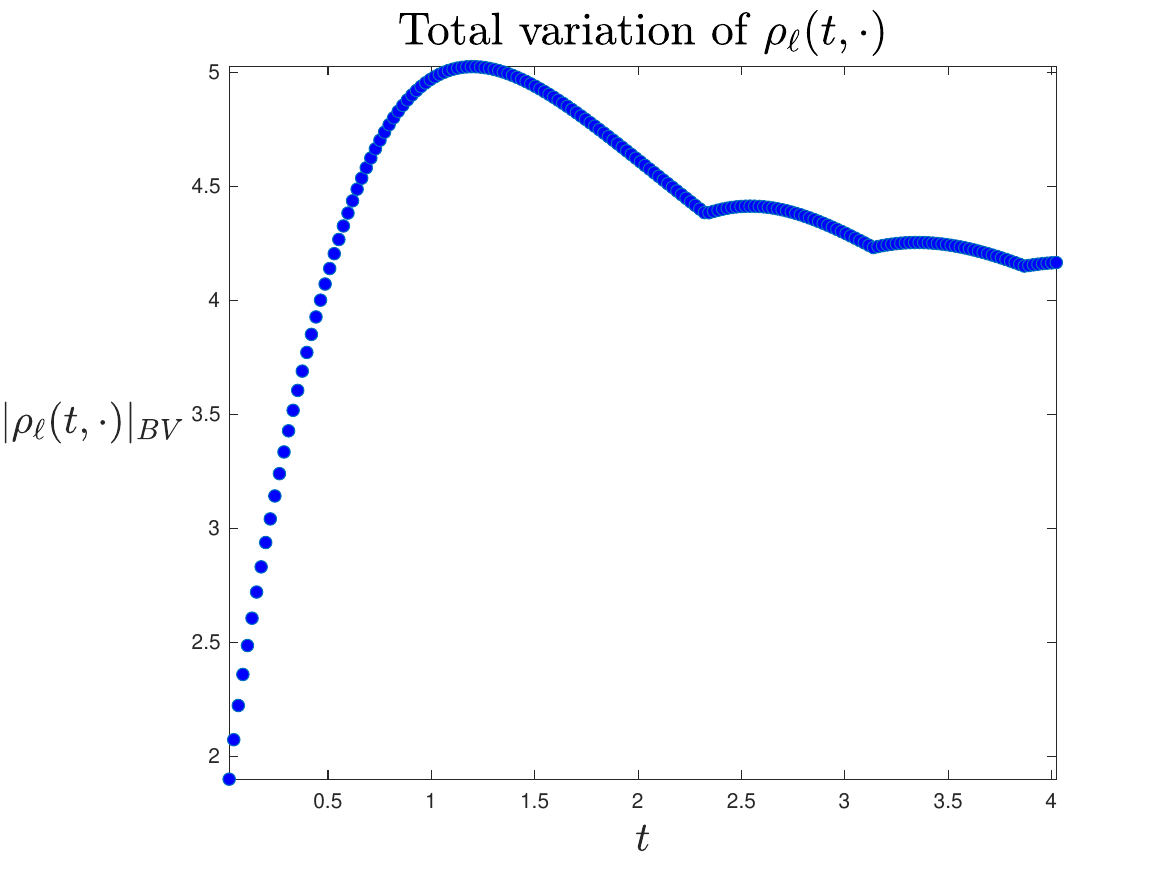}
    \end{tabular}
    \caption{Left: $\rho_\ell(1,x)$, the red circles at the bottom
      indicate the position of the vehicles. Right: The total
      variation of $\rho_\ell$ as a function of $t\in[0,4]$.}
    \label{fig:1}
  \end{figure}
  From Figure~\ref{fig:1} we see that $\rho_\ell$ develops large
  oscillations, and that it most likely is not true that
  $\abs{\rho_\ell(t,\dott)}_{BV}\le \abs{\rho_0}_{BV}$ for the model
  specified by \eqref{eq:MAIN_shortmodel} when $N>1$.
\end{remark*}

\begin{corollary}
  \label{cor:rhoconv}
  Under the same assumptions as in Lemma~\ref{lem:convergence}, for any
  $T>0$ there
  exists a function $\rho\in C([0,T];L^1([0,P]))\cap L^\infty([0,T];BV([0,P]))$ such
  that (up to a subsequence)
  \begin{equation*}
    \lim_{\ell\to 0} \rho_\ell = \rho.
  \end{equation*}
\end{corollary}

\section{The limit}\label{sec:limit}

Having established the existence of the limit of $\rho_\ell$ to a function $\rho$ as $\ell\to 0$ in Corollary \ref{cor:rhoconv}, we now need to show that the limit is the unique weak entropy solution of the LWR equation, that is, 
\begin{equation}\label{eq:kruzkov}
      \int_0^\infty\int_0^P \big(\eta(\rho)
    \test_t +q(\rho)\test_x\big) \,dxdt
      +\int_0^P \eta(\rho_0)\test(0,x)\,dx\ge 0
  \end{equation}
for all non-negative $P$-periodic test functions $\test\in C^\infty_c([0,\infty)\times[0,P])$.   Here $\eta$ is a convex (entropy) function and  $q$ is the entropy flux, satisfying $q'(\rho)=\eta'(\rho)(\rho v(\rho))'$. 

\begin{theorem}  \label{lem:weaksolN}
Let $v$ be a  Lipschitz continuous non-increasing velocity function $v\colon[0,1] \to [0,1]$, with $v(0)=1$ and  $v(1)=0$. 
Let  $P>0$ and $\rho_0\in L^1([0,P])\cap BV([0,P])$ be $P$-periodic such
that $\rho_0(x)\in [\nu,1]$ for all $x$ for some $\nu>0$.  Define
$x_i(0)$ by   \eqref{eq:x0def} for $i\in Z_\ell$, and let
$x_i(t)$ be defined by \eqref{eq:MAIN_shortmodel}.

Define $\rho_i$ for $i\in Z_\ell$ by  \eqref{eq:drhodef}, and  $\rho_\ell(t,x)$ by \eqref{eq:rho_ell_def} for  $\ell>0$.  Denote the limit as $\ell\to0$ of $\rho_\ell(t,x)$, granted by Corollary \ref{cor:rhoconv}, by $\rho$ with
 $\rho\in C([0,T];L^1([0,P]))\cap L^\infty([0,T];BV([0,P]))$.
 Then $\rho$ is the unique weak entropy solution
  of the scalar conservation law
\begin{equation}
\rho_t+(\rho v(\rho))_x=0, \quad \rho|_{t=0}= \rho_0.
\end{equation}  
\end{theorem}
\begin{proof}
  Let $(\eta,q)$ be an entropy/entropy flux pair.  We define
  \begin{align*}
    \eta_\ell(t,x)
    &=\sum_i \eta_i(t)\indic_i(t,x),\\
    \eta^\eps_\ell(t,x)
    &=\sum_i \eta_i(t)\chi^\eps_i(t,x),
  \end{align*}
  where, in the obvious notation, $\eta_i=\eta(\rho_i)$, and
  $\chi^\eps_i$ is given by \eqref{eq:chi_eps}.  The dynamics of
  $\eta_i$ is given by
  \begin{equation*}
    \dot{\eta}_i = -\eta'(\rho_i)\rho_i  \frac{\Dlp \ob{v}_i +\kappa \Dlp\Dlm v_i}{\Dlp  x_i}.
  \end{equation*}

  Fix a nonnegative test function $\test$ and calculate
  \begin{align}
    \int_0^\infty
    &\int_0^P
      \eta^\eps_\ell \test_t \,dxdt \notag\\
    &=-\int_0^P \eta^\eps_\ell(0,x)\test(0,x)\,dx -
      \int_0^\infty\int_0^P \pt\eta^\eps_\ell \test \,dxdt \notag\\
    &=-\int_0^P \eta^\eps_\ell(0,x)\test(0,x)\,dx -
      \int_0^\infty\int_0^P
      \sum_i  \Bigl(\dot{\eta}_i\chi^\eps_i + \eta_i \frac{\partial}{\partial t}\chi^\eps_i\Bigr)\test\,dxdt 
      \notag  \\
    &=-\int_0^P \eta^\eps_\ell(0,x)\test(0,x)\,dx \notag\\
    &\quad -
      \int_0^\infty\int_0^P
      \sum_i  \Bigl[\Bigl(-\eta'(\rho_i)\rho_i  \frac{\Dlp \ob{v}_i +\kappa \Dlp\Dlm v_i}{\Dlp  x_i}\Bigr)\chi^\eps_i \notag\\
    &\qquad\qquad+
      \eta_i \Dlp\left[ \omega_\eps(x-x_i)\left(\ob{v}_i+\kappa\Dlm v_i\right) \right]\Bigr]\test\,dxdt. \label{eq:weak_eps}
  \end{align}
  Next we want to take the $\eps\to0$ limit. To that end, we consider
  some of the terms separately.  We have
  \begin{align*}
    &\int_0^P
      \sum_i \Bigl(-\eta'(\rho_i)\rho_i  \frac{\Dlp \ob{v}_i +\kappa \Dlp\Dlm v_i}{\Dlp  x_i}\Bigr)\chi^\eps_i\test \,dx\\
    &\qquad \overset{\eps\to 0}{\longrightarrow}
      \sum_i \int_{x_i}^{x_{i+1}}\Bigl(-\eta'(\rho_i)\rho_i  \frac{\Dlp \ob{v}_i +\kappa \Dlp\Dlm v_i}{\Dlp  x_i}\Bigr)\test \,dx\\
    &\qquad\qquad= \sum_i \Bigl(-\eta'(\rho_i)\rho_i (\Dlp \ob{v}_i +\kappa \Dlp\Dlm v_i)\Bigr)\test_{i+1/2} 
  \end{align*}
  where
  \begin{equation*}
    \test_{i+1/2}=\frac1{\Dlp  x_i} \int_{x_i}^{x_{i+1}}\test\,dx.
  \end{equation*}
  Furthermore, we will use
  \begin{align*}
    &\int_0^P\eta_i \Dlp\left[ \omega_\eps(x-x_i)\left(\ob{v}_i+\kappa\Dlm v_i\right) \right]\test\,dx \\
    &\qquad=\int_0^P\eta_i \Dlp\left[ \omega_\eps(x-x_i)\left(\ob{v}_i+\kappa\Dlm v_i\right) \test\right]\,dx \\
    &\qquad \overset{\eps\to 0}{\longrightarrow}\eta_i \Dlp\left[ \left(\ob{v}_i+\kappa\Dlm v_i\right) \test_i\right],
  \end{align*}
  with $\test_i=\test(t,x_i(t))$.  Note that while $\test_{i+1/2}$ is
  defined by a spatial average, $\test_i$ is given as a pointwise
  value.  This yields that the limit as $\eps\to0$ in
  \eqref{eq:weak_eps} satisfies
  \begin{align*}
    \int_0^\infty
    &\int_0^P
      \eta_\ell \test_t \,dxdt\\
    &=-\int_0^P \eta_\ell(0,x)\test(0,x)\,dx \\
    &\quad -
      \int_0^\infty
      \Bigl(\sum_i \Bigl(-\eta'(\rho_i)\rho_i (\Dlp \ob{v}_i +\kappa \Dlp\Dlm v_i)\Bigr)\test_{i+1/2} \\
    &\qquad\qquad+
      \eta_i \Dlp\left[ \left(\ob{v}_i+\kappa\Dlm v_i\right)\phi_i \right]\Bigr)\,dt \\
    &=-\int_0^P \eta_\ell(0,x)\test(0,x)\,dx\\
    &\quad + \int_0^\infty \sum_i \Big(\rho_i \eta'(\rho_i) \Dlp
      \ob{v}_i\test_{i+1/2} - \eta_i
      \Dlp \ob{v}_{i}\test_{i+1}-\eta_i\ob{v}_i\Dlp\test_i \\
    &\qquad \hphantom{ + \int_0^\infty \sum_i} +\kappa\left(\Dlp\Dlm
      v_i\right)
      \left(\rho_i\eta'(\rho_i)\test_{i+1/2}-\eta_i\test_i\right)\Big)\,dt.
  \end{align*}
 
  Next we want to replace $\test_{i+1/2}$ and $\test_{i+1}$ with
  $\test_i$, and by doing so, we will introduce an error term of order
  $\bigOh(\ell)$.  First observe that
  \begin{align*}
    \abs{\test_{i+1/2}-\test_i} &\le \frac12
                                  \norm{\test_x}_{\infty}
                                  \Dlp x_i\le  \frac{\ell}{2\nu} \norm{\test_x}_{\infty}, \\ 
    \abs{ \Dlp\test_{i}}=  \abs{\test_{i}-\test_{i+1}} & \le 
                                                         \norm{\test_x}_{\infty}  \Dlp x_i\le  \frac{\ell}{\nu}\norm{\test_x}_{\infty},
  \end{align*}
  since
  \begin{equation*}
    \Dlp x_i(t)\le \ell \sup_i y_i(t)\le \frac{\ell}{\nu},
  \end{equation*}
  where
  $\norm{\test_x}_{\infty}=\norm{\test_x}_{L^\infty([0,\infty)\times[0,P])}$.

  For the relevant terms we add and subtract $\test_i$. Thus
  \begin{align*}
    \Bigl|\int_0^\infty \sum_i \rho_i \eta'(\rho_i) \Dlp
    \ob{v}_i(\test_{i+1/2}-\test_i)\,dt\Bigr| &\le \sup_{\rho\in[0,1]}\abs{\eta'(\rho)} T \sum_i \abs{\Dlp \ob{v}_i} \norm{\test_x}_{\infty}\frac\ell\nu\\
                                              &\le \sup_{\rho\in[0,1]}\abs{\eta'(\rho)} \frac{T\ell}\nu\norm{\test_x}_{\infty} \abs{\ob{v}}_{BV}\\
                                              &\le \sup_{\rho\in[0,1]}\abs{\eta'(\rho)} \frac{T\ell}\nu\norm{\test_x}_{\infty} \norm{v'}_{\infty}\abs{\rho_\ell(t,\dott)}_{BV}\\
                                              &\le \sup_{\rho\in[0,1]}\abs{\eta'(\rho)} \frac{T\ell}{\nu^3}\norm{\test_x}_{\infty} \norm{v'}_{\infty}\abs{\rho_0}_{BV}=\bigOh(\ell),
  \end{align*}
  where we used that $\rho_\ell\in[0,1]$, \eqref{eq:BVmean},
  \eqref{eq:rhobv}, and $T>0$ is such that $\test(t,x)=0$ for $t>T$
  and all $x$. Next, we find
  \begin{align*}
    \Bigl|\int_0^\infty \sum_i \eta_i
    \Dlp \ob{v}_{i}(\test_{i+1}-\test_i)\,dt\Bigr| &\le \sup_{\rho\in[0,1]}\abs{\eta(\rho)}  
                                                     \frac{T\ell}{\nu^3}\norm{\test_x}_{\infty} \norm{v'}_{\infty}\abs{\rho_0}_{BV}=\bigOh(\ell),
  \end{align*}
  by the same estimates as above.  Finally, we consider
  \begin{align*}
    \Bigl|\int_0^\infty \sum_i\left(\Dlp\Dlm
    v_i\right)
    \rho_i\eta'(\rho_i)(\test_{i+1/2}-\test_i)\,dt\Bigr| &\le  \sup_{\rho\in[0,1]}\abs{\eta'(\rho)} \frac{T\ell}{2\nu} \norm{\test_x}_{\infty} 2\abs{v}_{BV}\\
                                                         &\le   \sup_{\rho\in[0,1]}\abs{\eta'(\rho)}\norm{v'}_{\infty}\abs{\rho_0}_{BV} \frac{T\ell}{\nu}=\bigOh(\ell), 
  \end{align*}
  as in the estimates above.  Thus
  \begin{align}
    \int_0^\infty
    \int_0^P
    \eta_\ell \test_t \,dxdt
    &=-\int_0^P \eta_\ell(0,x)\test(0,x)\,dx \notag\\
    &\quad + \int_0^\infty \sum_i \Big(\rho_i \eta'(\rho_i) \Dlp
      \ob{v}_i\test_i - \eta_i
      \Dlp \ob{v}_{i}\test_i-\eta_i\ob{v}_i\Dlp\test_i\notag \\
    &\qquad \hphantom{ + \int_0^\infty \sum_i} +\kappa\left(\Dlp\Dlm
      v_i\right)
      \left(\rho_i\eta'(\rho_i)-\eta_i\right)\test_i\Big)\,dt+ {\mathcal O}(\ell).  \label{eq:nesten}
  \end{align}
  Define $h(\rho)=\rho \eta'(\rho)-\eta(\rho)$ and note that
  $h'(\rho)=\rho\eta''(\rho)\ge 0$. Using this, the last term above
  equals
  \begin{align*}
    \sum_i (\Dlp\Dlm v_i)\, \test_i h_i
    &=-\sum_i \Dlm v_i \Dlm(h_i\test_i) \\
    &=-\sum_i \Big(\Dlm v_i \Dlm h_i\,\test_{i-1} + h_i \Dlm v_i \Dlm\test_i\Big)\\
    &= -\sum_i \Dlm v_i \Dlm \rho_i h'(\hat\rho_i)\,\test_{i-1} +\bigOh(\ell)\\
    &\ge \bigOh(\ell),
  \end{align*}
  for some $\hat\rho_i$ between $\rho_i$ and $\rho_{i-1}$. We used the
  nonnegativity of the test function and that
  $\Dlm v_i \Dlm \rho_i\le 0$ since $v$ is non-increasing in $\rho$.

  We want to replace the term $\eta_i\ob{v}_i$ by $(\ob{\eta v})_i$,
  and this invokes an error of $\bigOh(\ell)$, as the following
  computation reveals.
  \begin{align}
    \Bigl|\int_0^\infty \sum_i \big(\eta_i\ob{v}_i-\ob{\eta v}_i\big)
    \Dlp\test_i\, dt\Bigr|
    &=\Bigl|\int_0^\infty \sum_i  \sum_{j=0}^N
      c_j\big(\eta_i-\eta_{i+j}\big)v_{i+j} \Dlp\test_i\, dt\Bigr| \notag\\
    &=\Bigl|\int_0^\infty \sum_i  \sum_{j=0}^N
      \sum_{k=0}^{j-1}c_j\big(\eta_{i+k}-\eta_{i+k+1}\big)v_{i+j}
      \Dlp\test_i\, dt\Bigr|\notag\\
    &\le  \norm{v}_\infty \int_0^T \sum_{j=0}^N c_j
      \sum_{k=0}^{j-1} \sum_i \abs{\eta_{i+k}-\eta_{i+k+1}} \abs{\Dlp\test_i}\, dt\notag\\
    &\le  \norm{v}_\infty \norm{\test_x}_\infty \frac{\ell}{\nu}
      \int_0^T \abs{\eta(t)}_{BV}\sum_{j=0}^N j c_j \, dt \notag\\
    &\le \norm{v}_\infty \norm{\test_x}_\infty
      \frac{\ell}{\nu}\norm{\eta'}_\infty
      \int_0^T \abs{\rho(t)}_{BV}\sum_{j=0}^N j c_j \, dt \notag\\
    &\le T \norm{v}_\infty \norm{\test_x}_\infty
      \frac{\ell}{\nu}\norm{\eta'}_\infty  \abs{\rho_0}_{BV}\sum_{j=0}^N j c_j \notag\\
    &=\bigOh(\ell).  \label{eq:eta_middel}
  \end{align}
  Thus we find that \eqref{eq:nesten} can be re-written as
  \begin{align}
    \int_0^\infty\int_0^P \eta_\ell \test_t \,dxdt
    &\ge -\int_0^P \eta_\ell(0,x)\test(0,x)\,dx \notag\\
    &\quad + \int_0^\infty \sum_i \left(\rho_i
      \eta'(\rho_i)-\eta_i\right)\Dlp
      \ob{v}_i\test_{i}\,dt-\int_0^\infty\int_0^P \left(\ob{\eta v}\right)_\ell
      \test_x\,dxdt+\bigOh(\ell).  \label{eq:eta_nesten}
  \end{align}
  In order to understand the convective term we compare this with the
  term involving the entropy flux.  We define
  \begin{equation*}
    q_\ell(t,x)=\sum_i q_i(t)\indic_i(t,x), \quad  q_i(t)=q(\rho_i(t)),
  \end{equation*}
  which implies that
  \begin{align}
    \int_0^\infty\int_0^P q_\ell\test_x\,dxdt
    = \int_0^\infty\sum_i  q_i  \Dlp \test_i \, dt &= \int_0^\infty\sum_i
                                                     \ob{q}_i\Dlp \test_i\,dt+\bigOh(\ell)\notag \\
                                                   & = -\int_0^\infty\sum_i \Dlp \ob{q}_i \test_i \,dt+  \bigOh(\ell),\label{eq:gensecondA}
  \end{align}
  where we have replaced $q_i$ by $\ob{q}_i$, resulting in an error of
  order $\bigOh(\ell)$, from the following computation
  \begin{align*}
    \Bigl|\int_0^\infty\sum_i  \big(q_i- \ob{q}_i\big)\Dlp \test_i \,dt  \Bigr|&=
                                                                                 \Bigl|\int_0^\infty\sum_i  \sum_{j=0}^N c_j\big(q_i- q_{i+j}\big)\Dlp \test_i \,dt  \Bigr| \\
                                                                               &\le \int_0^T\sum_i  \sum_{j=0}^N c_j\abs{q_i- q_{i+j}} \abs{\Dlp \test_i} \,dt  \\
                                                                               &\le T \norm{\test_x}_\infty \frac{\ell}{\nu}\norm{q'}_\infty  \abs{\rho_0}_{BV}\sum_{j=0}^N j c_j \\
                                                                               &=\bigOh(\ell),  
  \end{align*}    
  similarly to \eqref{eq:eta_middel}. The observant reader will also
  have noticed that we have replaced $\test_{i+1}$ by $\test_i$ in \eqref{eq:gensecondA}, which
  also yields another $\bigOh(\ell)$ error.

  Adding \eqref{eq:eta_nesten} and \eqref{eq:gensecondA} we get
  \begin{align*}
    \int_0^\infty\int_0^P& \big(\eta_\ell
                          \test_t +q_\ell\test_x\big) \,dxdt
                          +\int_0^P \eta_\ell(0,x)\test(0,x)\,dx\\                                 
                        &\ge \int_0^\infty \sum_i\Bigl[\underbrace{
                          \left(\rho_i \eta'(\rho_i)-\eta_i\right)\Dlp \ob{v}_i -
                          \Dlp\ob{q}_i}_{e_i}\Bigr]\test_i\,dt
                          -\int_0^\infty\int_0^P \left(\ob{\eta v}\right)_\ell
                          \test_x\,dxdt +  \bigOh(\ell),
  \end{align*} 
  and it remains to estimate the term $e_i$. To that end we find,
  using the formula \eqref{eq:byparts}, that
  \begin{align*}
    e_i
    &= -\sum_{j=1}^N \Dlm c_j \left[\left(\rho_i \eta'(\rho_i)-\eta_i\right)
      \left(v_{i+j}-v_i\right) - \left(q_{i+j}-q_i\right)\right]\\
    &= -\sum_{j=1}^N \Dlm c_j \int_{\rho_i}^{\rho_{i+j}} \Big(
      \left(\rho_i \eta'(\rho_i)-\eta_i\right) v'(s) - \eta'(s)s
      v'(s)-\eta'(s)v(s)\Big)\,ds\\
    &= -\sum_{j=1}^N \Dlm c_j \int_{\rho_i}^{\rho_{i+j}}\Big(
      v'(s)\left(\rho_i\eta'(\rho_i)-s\eta'(s)\right)
      -\eta_i'v'(s)-\eta'(s)v(s)\Big)\,ds\\
    &= -\sum_{j=1}^N \Dlm c_j \int_{\rho_i}^{\rho_{i+j}} \Big(v'(s)
      \int_s^{\rho_i} \frac{d}{d\sigma}(\sigma \eta'(\sigma))\,d\sigma
      -\eta'(\rho_i)v'(s)-\eta'(s)v(s)\Big)\,ds\\
    &= -\sum_{j=1}^N \Dlm c_j \int_{\rho_i}^{\rho_{i+j}} \Big(v'(s)
      \int_s^{\rho_i} \sigma\eta''(\sigma)+\eta'(\sigma)\,d\sigma
      -\eta'(\rho_i)v'(s)-\eta'(s)v(s)\Big)\,ds\\
    &=\sum_{j=1}^N\Dlm c_j \int_{\rho_i}^{\rho_{i+j}}  \int_{\rho_i}^s
      v'(s)\sigma\eta''(\sigma) \,d\sigma ds\\
    &\quad  -\sum_{j=1}^N \Dlm c_j
      \int_{\rho_i}^{\rho_{i+j}}\big(
      v'(s)(\eta_i-\eta(s))-\eta_iv'(s)
      -\eta'(s)v(s)\big)\,ds\\
    &\ge \sum_{j=1}^N\Dlm c_j \int_{\rho_i}^{\rho_{i+j}}\big(
      v'(s)\eta(s)+v(s)\eta'(s)\big)\,ds\\
    &= \sum_{j=1}^N\Dlm c_j \left(v_{i+j}\eta_{i+j}-v_i\eta_i\right)\\
    &=-\Dlp \left(\ob{v\eta}\right)_i.
  \end{align*}
  In the inequality we used that $\Dlm c_j v'(s)\ge 0$, as both $c_j$
  and $v$ are non-increasing, $\rho_i$ and $\rho_{i+j}$ are positive,
  and finally $\eta''(\sigma)\ge 0$.
  
  Therefore
  \begin{align*}
    \int_0^\infty\int_0^P \big(\eta_\ell
    &\test_t +q_\ell\test_x\big) \,dxdt
      +\int_0^P \eta_\ell(0,x)\test(0,x)\,dx\\                                 
    &\ge \int_0^\infty \sum_i e_i\test_i\,dt
      -\int_0^\infty\int_0^P \left(\ob{\eta v}\right)_\ell
      \test_x\,dxdt +  \bigOh(\ell)\\
    &\ge - \int_0^\infty \sum_i \Dlp \left(\ob{v\eta}\right)_i\test_i\,dt
      -\int_0^\infty\int_0^P \left(\ob{\eta v}\right)_\ell
      \test_x\,dxdt +  \bigOh(\ell)\\
    &= \int_0^\infty \sum_i \left(\ob{v\eta}\right)_i\Dlm \test_i\,dt \quad 
      -\int_0^\infty\int_0^P \left(\ob{\eta v}\right)_\ell
      \test_x\,dxdt +  \bigOh(\ell)\\
    &=\bigOh(\ell),
  \end{align*}
  and by taking $\ell\to0$, we obtain \eqref{eq:kruzkov}. By standard
  theory, see, e.g., \cite[Thm.~2.14]{HoldenRisebro}, it follows that
  $\rho$ is the unique weak entropy solution.
\end{proof}


\end{document}